\documentclass[12pt]{article}
\usepackage{latexsym, amscd, amsfonts, eucal, mathrsfs, amsmath, mathabx, amssymb, amsthm, xypic,xr, stmaryrd, color, enumerate, tikz}
\usepackage{appendix}
\usepackage{multicol}
\usepackage[all]{xy}
\usepackage{hyperref}
\usepackage{fullpage}

\newtheorem{thm}{Theorem}
\newtheorem*{cor}{Corollary}

\newtheorem{theorem}{Theorem}[section]
\newtheorem{lemma}[theorem]{Lemma}
\newtheorem{proposition}[theorem]{Proposition}

\theoremstyle{definition}

\newtheorem{construction}[theorem]{Construction}
\newtheorem*{cstr}{Construction}

\newtheorem{question}{Question}

\newtheorem{remark}[theorem]{Remark}
\newtheorem*{rmk}{Remark}

\usepackage{cleveref}
\usepackage{scrextend}

\begin{document}

\title{On the $C_p$-equivariant dual Steenrod algebra}
\author{Krishanu Sankar and Dylan Wilson}
\date{}
\maketitle

\begin{abstract}
We compute the
$C_p$-equivariant dual Steenrod algebras associated to the
constant Mackey functors $\underline{\mathbb{F}}_p$ and $\underline{\mathbb{Z}}_{(p)}$, as 
$\underline{\mathbb{Z}}_{(p)}$-modules. The $C_p$-spectrum
$\underline{\mathbb{F}}_p \otimes \underline{\mathbb{F}}_p$
is not a direct sum of $RO(C_p)$-graded suspensions
of $\underline{\mathbb{F}}_p$ when $p$ is odd, in contrast
with the classical and $C_2$-equivariant dual Steenrod algebras.
\end{abstract}

\tableofcontents

\section*{Introduction}

For over a decade, since the Hill-Hopkins-Ravenel solution
of the Kervaire invariant one problem \cite{hhr}, there has been
great success in using exotic homotopy theories, like
$C_{2^n}$-equivariant homotopy theory and motivic homotopy theory,
to study classical homotopy theory at the prime 2.
A key foundational input to many of these applications
is the computation of the appropriate version of the dual
Steenrod algebra, $\underline{\mathbb{F}}_2 \otimes
\underline{\mathbb{F}}_2$, which was carried out by
Hu-Kriz \cite{hu-kriz} in $C_2$-equivariant homotopy theory and by
Voevodsky \cite{voevodsky} in motivic homotopy theory.
One of the major obstacles to carrying out a similar program at odd primes
is that we do not understand the structure
of the dual Steenrod algebra in $C_p$-equivariant homotopy theory.
The purpose of this paper is to make some progress towards this goal.

To motivate the statement of our main result, recall that
we have the following description of the
classical, $p$-local dual Steenrod algebra
as a $\mathbb{Z}_{(p)}$-algebra\footnote{We learned this fact from John Rognes.
One proof is to base change the equivalence \\$\mathrm{BP}\otimes_{S^0[v_1, ...]}S^0
\simeq \mathbb{Z}_{(p)}$ to $\mathbb{Z}_{(p)}$ and use that the Hurewicz image
of the $v_i$'s are $pt_i$, mod decomposables.}

\[
\mathbb{Z}_{(p)} \otimes \mathbb{Z}_{(p)}
\simeq
\mathbb{Z}_{(p)} \otimes 
\bigotimes_i
\mathrm{cofib}\left( \Sigma^{|t_i|}S^0[t_i]
\stackrel{\cdot pt_i}{\longrightarrow} S^0[t_i]\right).
\]
Here the tensor product is taken over the sphere spectrum,
$S^0[x]$ denotes the free $\mathbb{E}_1$-algebra on
a class $x$,
and the classes $t_i$ live in degree $2p^i-2$.
Modding out by $p$ causes each of the above cofibers to split
into two classes related by a Bockstein; modding out by
$p$ once more introduces the class $\tau_0$ and recovers
Milnor's computation of $\mathcal{A}_* = \pi_*(\mathbb{F}_p
\otimes \mathbb{F}_p)$, as an $\mathbb{F}_p$-algebra.

In the $C_p$-equivariant case our description involves a similar
decomposition but is more complicated in two ways:
	\begin{itemize}
	\item Rather than extending the class
	$t_i$ to a map from $S^0[t_i]$ using the multiplication
	on $\mathbb{Z}\otimes \mathbb{Z}$, we will
	want to choose as generators a mixture of
	\emph{ordinary powers} of $t_i$ and of
	\emph{norms}, $N(t_i)$, of $t_i$.
	\item Rather than modding out by the relation
	`$pt_i=0$' we will need to enforce the relation
	that `$\theta t_i = 0$', where $\theta$ is
	an equivariant
	lift of $p$ to an element in nontrivial $RO(C_p)$-degree.
	We will then \emph{also} need to enforce the relation
	$pN(t_i) = 0$. 
	\end{itemize}
	
To make this precise, we will assume that the reader
is comfortable with equivariant stable homotopy theory
as used, for example, in \cite{hhr}, and introduce the following
conventions, in force throughout the paper:
	\begin{itemize}
	\item We will use $\rho_{C_p}$ to denote
	the regular representation of $C_p$.
	\item We will use $\lambda$ to denote the
	representation of $C_p$ on $\mathbb{R}^2=\mathbb{C}$
	where the generator acts by $e^{2\pi i/p}$.
	\item We denote by $\theta: S^{\lambda - 2} \to S^0$
	the map of $C_p$-spectra arising from the degree
	$p$ cover $S^{\lambda} \to S^2$. We'll denote the
	cofiber of $\theta$ by $C\theta$. Note that
	the underlying nonequivariant spectrum of $C\theta$ is the
	Moore space $M(p)$. 
	\item If $X$ is a spectrum, we will denote by
	$N(X)$ the Hill-Hopkins-Ravenel norm of $X$,
	which is a $C_p$-equivariant refinement of
	the ordinary spectrum $X^{\otimes p}$.
	\item We denote by $\underline{\mathbb{Z}}$ and
	$\underline{\mathbb{F}}_p$ the $C_p$-equivariant
	Eilenberg-MacLane spectra associated to the constant
	Mackey functors at $\mathbb{Z}$ and $\mathbb{F}_p$,
	respectively.
	\item We use $\otimes$ to denote the symmetric monoidal
	structure on genuine $C_p$-spectra, $\mathsf{Sp}^{C_p}$ (often denoted
	by $\wedge$, the smash product).
	\item The degree $k$ map
	\[
	S^{\lambda} \to S^{\lambda^k}
	\]
	is a $p$-local equivalence when $(k, p) = 1$, so, when
	working $p$-locally, we will often make this identification
	implicitly. For example, we may write
	\[
	S^{\rho_{C_p}}_{(p)} = S_{(p)}^{1+\frac{p-1}{2} \lambda}.
	\]
	\item We use $\pi_{\star}X$ to denote the $RO(C_p)$-graded
	homotopy groups of a $C_p$ spectrum, so that, when
	$\star = V-W$ is a virtual representation,
	$\pi_{V-W}X = \pi_0\mathrm{Map}_{\mathsf{Sp}^{C_p}}(S^{V-W}, X)$.
	\end{itemize}

Now we can give a somewhat ad-hoc description of the equivariant
refinements of the building blocks in $\mathbb{Z} \otimes \mathbb{Z}$.

\begin{cstr}
Let $x$ be a formal variable in an $RO(C_p)$-grading
$|x|$. Define a $C_p$-spectrum as follows:
	\[
	T_{\theta}(x):=
	\Sigma^{|x|}C\theta
	\oplus
	\Sigma^{2|x|}C\theta
	\oplus
	\cdots
	\oplus
	\Sigma^{(p-1)|x|}C\theta
	\oplus
	\Sigma^{|x|\rho_{C_p}}M(p),
	\]
where $M(p)$ is the mod $p$ Moore spectrum.
We denote the inclusion of the summand
$\Sigma^{i|x|}C\theta$ by
	\[
	x^{i-1}\hat{x}: \Sigma^{i|x|}C\theta \to T_{\theta}(x),
	\]
the restriction of $\hat{x}$ to the bottom cell by
$x$,
and the inclusion of the final summand by
$\widehat{Nx}$. We denote by
	\[
	Nx: S^{|x|\rho_{C_p}} \to T_{\theta}(x)
	\]
 the restriction of $\widehat{Nx}$ to the bottom
 cell of the mod $p$ Moore spectrum.
 
Now suppose that $R$ is a $C_p$-ring spectrum equipped with a norm
$N(R) \to R$. If we have a class $x \in \pi_{\star}R$ such that
$\theta x = 0$, it follows that $p\cdot N(x) = 0$ (see the proof of
Lemma \ref{lem:pn=0}),
so we may produce a map
	\[
	S^0 \oplus (S^0[Nx] \otimes T_{\theta}(x)) \to R,
	\]
which only depends on the choice of the nullhomotopy witnessing
$\theta x = 0$. 
\end{cstr}

We can now state our main theorem.

\begin{thm}\label{thm:main} There are equivariant refinements
	\[
	t_i: S^{2p^{i-1}\rho_{C_p} - \lambda} \to
	\underline{\mathbb{Z}}_{(p)} \otimes \underline{\mathbb{Z}}_{(p)}
	\]
of the nonequivariant classes $t_i \in 
\pi_*(\mathbb{Z}_{(p)}\otimes \mathbb{Z}_{(p)})$
which satisfy the relation
$\theta t_i = 0$. 
For any choice of witness for these relations, the
resulting map
	\[
	\underline{\mathbb{Z}}_{(p)} \otimes
	\bigotimes_{i\ge 1}
	\left( S^0 \oplus (S^0[Nt_i] \otimes T_{\theta}(t_i))\right)
	\longrightarrow 
	\underline{\mathbb{Z}}_{(p)} \otimes \underline{\mathbb{Z}}_{(p)}
	\]
is an equivalence. 
\end{thm}

As an immediate corollary we have:

\begin{cor} With notation as above, we have
	\[
	\underline{\mathbb{F}}_p \otimes
	\underline{\mathbb{F}}_p
	\simeq
	\Lambda(\tau_0) \otimes_{\underline{\mathbb{F}}_p}
	\underline{\mathbb{F}}_{p} \otimes
	\bigotimes_{i\ge 1} \left( S^0 \oplus (S^0[Nt_i] \otimes T_{\theta}(t_i))\right),
	\]
where $\tau_0$ is dual to the Bockstein, in degree $1$ and
$\Lambda(\tau_0) = 
\underline{\mathbb{F}}_p \oplus \Sigma \underline{\mathbb{F}}_p$.
In particular, since 
$\underline{\mathbb{F}}_p \otimes C\theta$ is indecomposable
at odd primes, the spectrum $\underline{\mathbb{F}}_p \otimes
\underline{\mathbb{F}}_p$ is \emph{not} a direct sum of
$RO(C_p)$-graded suspensions of $\underline{\mathbb{F}}_p$
at odd primes.
\end{cor}

\begin{rmk} When $p=2$ we have an
accidental splitting $\underline{\mathbb{F}}_2 \otimes
C\theta \simeq \Sigma^{\sigma -1}\underline{\mathbb{F}}_2
\oplus \Sigma^{\sigma}\underline{\mathbb{F}}_2$, where
$\sigma$ is the sign representation.
\end{rmk}

\begin{rmk} One can show that
$\underline{\mathbb{F}}_p \otimes C\theta \otimes C\theta$ splits
as $(\underline{\mathbb{F}}_p\otimes C\theta) \oplus 
(\underline{\mathbb{F}}_p \otimes \Sigma^{\lambda -1}C\theta)$.
It follows that $\underline{\mathbb{F}}_p \otimes \underline{\mathbb{F}}_p$
splits as a direct sum of cell complexes with at most 2 cells.
\end{rmk}

Our result raises a few natural questions which would
be interesting to investigate.

\begin{question} When specialized to $p=2$, how does
our basis compare to the Hu-Kriz basis?
\end{question}

\begin{question} The geometric fixed points of
$\underline{\mathbb{Z}}_{(p)} \otimes \underline{\mathbb{Z}}_{(p)}$
are given by $(\mathbb{F}_p \otimes \mathbb{F}_p)[b, \overline{b}]$,
where $\overline{b}$ is the conjugate of $b$, a class in degree $2$.
It is possible to understand what happens to the generators
$t_i$ and $\widehat{N(t_i)}$ upon taking geometric fixed points.
One is left with trying
to understand the remaining class hit by $\hat{t}_i$ on geometric
fixed points. We don't know what this should be. One
guess that seems consistent with computations is that this class
is given, up to conjugating the $\tau_i$ and modding out by $(b)$, by:
	\[
	\tau_{i-1} + \overline{b}^{p^{i-1}-p^{i-2}} \tau_{i-2} + \cdots + 
	\overline{b}^{p^{i-1}-1}\tau_0
	\]
It would be useful for computations to sort out what actually occurs.
\end{question}

\begin{question} Is it possible to profitably study the
$\underline{\mathbb{F}}_p$-based Adams spectral
sequence using this decomposition? Since $\underline{\mathbb{F}}_p \otimes
\underline{\mathbb{F}}_p$ is not flat over $\underline{\mathbb{F}}_p$,
one would be forced to start with the $E_1$-term. But this is not
an unprecedented situation (e.g. Mahowald had
great success with the $\mathrm{ko}$-based Adams spectral sequence).
\end{question}

\begin{question} Can one describe the multiplication on
$\pi_{\star}\underline{\mathbb{F}}_p \otimes \underline{\mathbb{F}}_p$
in terms of our decomposition?
\end{question}

\subsubsection*{Relation to other work}

As we mentioned before, we were very much motivated
by the description of the $C_2$-equivariant dual Steenrod
algebra given by Hu-Kriz \cite{hu-kriz}. That said,
our generators are slightly different than the Hu-Kriz
generators when we specialize to $p=2$. For example,
the generator $t_1$ lives in degree $2\rho_{C_2}-\lambda
=2$, whereas the Hu-Kriz generator $\xi_1$ lives in degree
$\rho = 1+\sigma$.\footnote{In this low degree, it seems likely that,
modulo decomposables,
we have $u_{\sigma}\xi_1 = t_1$ and that $\xi_1$
is recovered from $\hat{t}_1$ by restricting along
$\underline{\mathbb{F}}_2 \otimes S^{1+\sigma} \to 
\underline{\mathbb{F}}_2 \otimes \Sigma^{2}C\theta$.}
Hill and Hopkins have also obtained a
presentation of the $C_{2^n}$-dual Steenrod
algebra, using quotients of $\mathrm{BP}\mathbf{R}$ and its norms,
which is similar in style to the one obtained here. 

At odd primes, Caruso \cite{caruso}
studied the $C_p$-equivariant Steenrod algebra,
$\pi_{\star}\mathrm{map}(\mathbb{F}_p, \mathbb{F}_p)$,
essentially by comparing
with the Borel equivariant Steenrod algebra and the
geometric fixed point Steenrod algebra, and was able
to compute the ranks of the integer-graded stems. 
There is also work of Oru\c{c} \cite{oruc} computing
the dual Steenrod algebra for the Eilenberg-MacLane
spectra associated to \emph{Mackey fields}
(which does not include $\underline{\mathbb{F}}_p$).

In the \emph{Borel} equivariant setting, the dual Steenrod
algebra is given by the action Hopf algebroid for the
coaction of the classical dual Steenrod
algebra on $H^*(\mathrm{B}C_p)$ (see
\cite{greenlees}).

There is also related work from the first and second
authors. The first author produced a splitting
of $\underline{\mathbb{F}}_p \otimes \underline{\mathbb{F}}_p$
in \cite{sankar} using the symmetric power filtration.
This summands in that splitting were roughly given by
the homology of classifying spaces, and were much
larger than the summands produced here. 
The second author and Jeremy Hahn
showed \cite{hahn-wilson} 
that $\underline{\mathbb{F}}_p$ can be obtained
as a Thom spectrum on $\Omega^{\lambda}S^{\lambda+1}$. 
The Thom isomorphism then reduces the study of
the dual Steenrod algebra to the computation of the homology
of $\Omega^{\lambda}S^{\lambda+1}$. Understanding
the relationship between this picture and the one
in this article is work in progress.

\subsubsection*{Acknowledgements}
We thank Jeremy Hahn and Clover
May for excellent discussions over the course of
this project. We thank Mike Hill for generously sharing
his time and his ideas, especially the idea to use
the norm to build some of the generators of the dual
Steenrod algebra. We apologize to those
awaiting the publication of these results for the long
delay.

\section{Outline of the proof}

To motivate our method of proof, let's first revisit the classical
story. We are interested in where the classes
$t_i \in \pi_*(\mathbb{Z} \otimes \mathbb{Z})$ come from,
and why they are annihilated by $p$.

Recall that the homology of $\mathbb{C}P^{\infty}$
is a divided power algebra
	\[
	H_*(\mathbb{C}P^{\infty}) = \Gamma_{\mathbb{Z}}\{\beta_1\}
	\]
where $\beta_1$ is dual to the first Chern class $c_1$. 
Write $\beta_{(i)}:=\gamma_{p^i}(\beta_1)$. 
Since $\mathbb{C}P^{\infty}=K(\mathbb{Z},2)$, we have
a map of spectra
	\[
	\mathbb{C}P^{\infty}_+ \to \Sigma^2\mathbb{Z}
	\]
and hence a homology suspension map
	\[
	\sigma: H_*(\mathbb{C}P^{\infty}) \to
	\pi_{*-2}(\mathbb{Z} \otimes \mathbb{Z})
	\]
which annihilates elements decomposable with
respect to the product structure on $H_*(\mathbb{C}P^{\infty})$.
We can take\footnote{Depending on ones preferences,
this might be the \emph{conjugate} of the
generator you want; but we are only really concerned
with these classes modulo decomposables.}
$t_i := \sigma(\beta_{(i)})$. The relation $pt_i = 0$ follows
from the fact that $p\beta_{(i)}$ is, up to a $p$-local unit,
decomposable as $\beta_{(i-1)}^{p}$ in $H_*(\mathbb{C}P^{\infty})$.

In the equivariant case, we will proceed similarly.

\begin{itemize}
\item[\underline{Step 1}.] Compute the homology of
$K(\underline{\mathbb{Z}}, \lambda)$ and use the homology
suspension to
define classes in $\pi_{\star}(\underline{\mathbb{Z}} \otimes
\underline{\mathbb{Z}})$.
\item[\underline{Step 2}.] Use information about the
product structure on the homologies of
$K(\underline{\mathbb{Z}}, \lambda)$
and $K(\underline{\mathbb{Z}}, 2)$ to deduce relations
for these classes, and hence produce the map
described in Theorem \ref{thm:main}.
\item[\underline{Step 3}.] Verify that the map
in Theorem \ref{thm:main} is an equivalence
by proving that it is an underlying equivalence and
an equivalence on geometric fixed points.
\end{itemize}

The first step is carried out in \S\ref{sec:homology-cpinfty} and
\S\ref{sec:homology-suspend} by identifying 
$K(\underline{\mathbb{Z}},\lambda)$ with an equivariant version
of $\mathbb{C}P^{\infty}$ and then specializing a computation
due to Lewis \cite{lewis}, which we review in our context.
The second step is carried out in \S\ref{sec:relations}.
The third and final step is carried out in \S\ref{sec:proof}
using a lemma proven in \S\ref{sec:detect-equivalences}
that allows us to check that the map on geometric
fixed points is an equivalence by just verifying that the
source and target have the same dimensions in each degree.

\section{Homology of $\mathrm{B}_{C_p}S^1$}\label{sec:homology-cpinfty}

Recall that we have the $C_p$-space $\mathrm{B}_{C_p}S^1$
classifying equivariant principal $S^1$-bundles.
The following lemmas give two useful ways of thinking about this
space.

\begin{lemma} The complex projective space
$\mathbb{P}(\mathbb{C}[z])$
is a model for $\mathrm{B}_{C_p}S^1$,
where the generator
of $C_p$ acts on $\mathbb{C}[z]$ through ring maps by
$z\mapsto e^{2\pi i/p}z$. Here $\mathbb{C}[z]$ is the ordinary
polynomial ring over $\mathbb{C}$, and the projective space
$\mathbb{P}(\mathbb{C}[z]) = (\mathbb{C}[z] - \{0\})/(\mathbb{C}^{\times})$
inherits an action in the evident way.
\end{lemma}

\begin{lemma} The space $\mathrm{B}_{C_p}S^1$ is
a model for $K(\underline{\mathbb{Z}}, \lambda)$.
\end{lemma}
\begin{proof} The map
	\[
	\mathbb{P}(\mathbb{C}[z]) \to \mathrm{SP}^{\infty}(S^{\lambda})
	\]
to the infinite symmetric product,
which sends a polynomial $f(z)$ to its set of roots (with multiplicity),
is an equivariant homeomorphism. The group-completion of
the latter is a model for $K(\underline{\mathbb{Z}},\lambda)$
by the equivariant Dold-Thom theorem. But 
$\mathrm{SP}^{\infty}(S^{\lambda})$ is already group-complete:
the monoid of connected components of the fixed points is
$\mathbb{N}/p = \mathbb{Z}/p$.
\end{proof}

\begin{remark} The reader may object that 
the definition of $\mathrm{B}_{C_p}S^1$ makes no reference
to $\lambda$, so how does $\mathrm{B}_{C_p}S^1$
know about this representation rather than $\lambda^k$ for 
some $k$ coprime to $p$? The answer is that, in fact,
each of the Eilenberg-MacLane spaces $K(\underline{\mathbb{Z}},
\lambda^k)$ coincide for such $k$: we have
an equivalence of $\underline{\mathbb{Z}}$-modules
	\[
	\Sigma^{\lambda}\underline{\mathbb{Z}} \simeq
	\Sigma^{\lambda^k}\underline{\mathbb{Z}}
	\]
whenever $(k,p)=1$. This follows from the computations
in \cite[Proposition 9.2]{ferland-lewis}, for example.
\end{remark}

The filtration of $\mathbb{C}[z]$ by the subspaces
$\mathbb{C}[z]_{\le n}$ of polynomials of degree at most $n$
gives a filtration of $\mathrm{B}_{C_p}S^1$.

\begin{lemma} There is a canonical equivalence
	\[
	\mathrm{gr}_k\mathrm{B}_{C_p}S^1
	\cong S^{V_k}.
	\]
where $V_k = \bigoplus_{0\le i \le k-1}\lambda^{i-k}$.
\end{lemma}
\begin{proof} This follows from a more general observation.
If $L$ is a one-dimensional 
complex representation, and $V$ is an arbitrary
complex representation, then the function
assigning a linear map to its graph,
	\[
	\mathrm{Hom}_{\mathbb{C}}(L,V) \longrightarrow
	\mathbb{P}(V\oplus L) - \mathbb{P}(V), 
	\]
is an equivariant homeomorphism. So it induces
an equivalence on one-point compactifications
	\[
	S^{L^{\vee}\otimes V} \cong \mathbb{P}(V\oplus L)/\mathbb{P}(V).
	\]
\end{proof}

The next proposition now follows from \cite[Proposition 3.1]{lewis}.

\begin{proposition}[Lewis] 
The above filtration on $\mathrm{B}_{C_p}S^1$ splits
after tensoring with $\underline{\mathbb{Z}}$,
giving an equivalence
	\[
	\underline{\mathbb{Z}} \otimes \mathrm{B}_{C_p}S^1_+
	\simeq
	\underline{\mathbb{Z}}\{e_0, e_1, ...\}
	\]
where
	\[
	|e_k| = \bigoplus_{0\le i \le k-1} \lambda^{i-k}.
	\]
In particular, for $i\ge 1$ we have $|e_{p^i}|= 2p^{i-1}\rho_{C_p}$.
\end{proposition}

We will also need some information about the multiplicative
structure on homology.

\begin{lemma}\label{lem:mult-in-bs1}
Writing $x\doteq y$ to mean that $x=\alpha y$
for some $\alpha \in \mathbb{Z}_{(p)}^{\times}$, we have
	\[
	e_{1}^p \doteq \theta e_p, \text{ and } e_{p^i}^p \doteq
	pe_{p^{i+1}}\text{ for }i\ge1.
	\]
\end{lemma}
\begin{proof} Using the model for $\mathrm{B}_{C_p}S^1$
given by $\mathbb{P}(\mathbb{C}[z])$, we see that,
in fact, $\mathbb{P}(\mathbb{C}[z])$ has the structure
of a filtered monoid. It follows that the product in
homology respects the filtration by the classes
$\{e_i\}$. Thus, for $i\ge 0$, we have:
	\[
	e_{p^i}^p = \sum_{j\le p^{i+1}} c_{i,j}e_j
	\]
where the coefficients lie in $\pi_{\star}\underline{\mathbb{Z}}$.
When $j<p^{i+1}$ we see that the virtual representations
$|c_{i,j}|$ have positive
virtual dimension and their fixed points also have 
positive virtual dimension. The homotopy of $\underline{\mathbb{Z}}$
vanishes in these degrees (see, e.g., \cite[Theorem 8.1(iv)]{ferland-lewis}),
so we must have
	\[
	e^p_{p^i} = c_{i, p^{i+1}}e_{p^{i+1}}
	\]
where $|c_{0, p}| = \lambda -2$ and
$|c_{i, p^{i+1}}| = 0$ when $i\ge 1$. In both cases,
the restriction map on $\pi_{\star}\underline{\mathbb{Z}}$
is injective in this degree, so the result follows
from the nonequivariant calculation.
\end{proof}

\section{Suspending classes}\label{sec:homology-suspend}

We begin with some generalities. If $X$ is any $C_p$-spectrum,
we have the counit
	\[
	\Sigma^{\infty}_+\Omega^{\infty} X \to X
	\]
which induces a map
	\[
	\sigma: \underline{\mathbb{Z}} \otimes 
	\Sigma^{\infty}_+\Omega^{\infty}X \to \underline{\mathbb{Z}} \otimes X,
	\]
called the homology suspension.
Just as in the classical case, $\sigma$ annihilates 
decomposable elements in $\pi_*(\underline{\mathbb{Z}} \otimes
\Sigma^{\infty}_+\Omega^{\infty}X)$.

\begin{construction} For $i\ge 1$, we define
	\[
	t_i: S^{2p^{i-1}\rho_{C_p}-\lambda} \to \underline{\mathbb{Z}} \otimes
	\underline{\mathbb{Z}}
	\]
as the homology suspension of the element
$e_{p^i} \in \pi_{2p^{i-1}\rho_{C_p}}(\underline{\mathbb{Z}}\otimes
\mathrm{B}_{C_p}S^1)$. 
Here we use the identification
	\[
	\mathrm{B}_{C_p}S^1 \simeq K(\underline{\mathbb{Z}},\lambda)
	= \Omega^{\infty}\Sigma^{\lambda}\underline{\mathbb{Z}}.
	\]
\end{construction}

\section{Two relations in homology}\label{sec:relations}

We begin with a brief review of norms, transfers, and restrictions.

\begin{remark}[Transfer and restriction]
Given a
nonequivariant equivalence $(S^V)^e \cong S^n$, we
define
	\[
	\mathrm{res}: \pi_VX \to \pi_nX^e, \quad
	(x: S^V \to  X) \mapsto (S^n \cong (S^V)^e \to X)
	\]
and
	\[
	\mathrm{tr}_V: \pi_nX^e \to \pi_VX, \quad
	(y: S^n \to X^e) \mapsto
	(S^V \to C_{p+} \otimes S^V \cong C_{p+} \otimes S^n
	\to C_{p+} \otimes X \to X).
	\]
For example, when $V=\lambda-2$ and
$X = S^0$, then $\mathrm{tr}_{\lambda -2}(1)
=\theta$.

Changing the equivalence $(S^V)^e \cong S^n$ has the
effect of altering these classes by $\pm 1$; in our case
the representations in question have canonical orientations
so this will not be a concern. Given a map $X \otimes Y \to Y$
we have a relation:
	\[
	\mathrm{tr}(x \otimes \mathrm{res}(y)) = \mathrm{tr}(x) \otimes y.
	\]
\end{remark}

\begin{remark}[Norms]
If a $C_p$-spectrum $X$ has a map $N(X) \to X$,
then, given an underlying class $x: S^n \to X^e$, we may define
a \textbf{norm} by the composite
	\[
	Nx: N(S^n) = S^{n\rho_{C_p}} \to N(X) \to X.
	\]
The underlying nonequivariant class is given by
$\mathrm{res}(Nx) = \prod_{g \in C_p} (gx) \in \pi_{pn}X^e$.
\end{remark}

Our goal in this section is to prove the following two lemmas.

\begin{lemma}\label{lem:thetat=0} The classes 
$t_i \in \pi_{2p^{i-1}\rho_{C_p}-\lambda}(\underline{\mathbb{Z}}_{(p)}
\otimes \underline{\mathbb{Z}}_{(p)})$
satisfy $\theta t_i = 0$.
\end{lemma}

\begin{lemma}\label{lem:pn=0} The classes $N(t_i)\in
\pi_{(2p^i-2)\rho_{C_p}}(\underline{\mathbb{Z}}_{(p)}
\otimes \underline{\mathbb{Z}}_{(p)})$
 satisfy $pN(t_i)=0$.
\end{lemma}

In fact, the second relation follows from the first.

\begin{proof}[Proof of Lemma \ref{lem:pn=0}
assuming Lemma \ref{lem:thetat=0}]
Since $p = \mathrm{tr}(1)$, the class $pN(t_i)$ is
the transfer of the class $\mathrm{res}(t_i)^p$
into degree $(2p^i-2)\rho_{C_p}$.
Notice that $(2p^i-2)\rho_{C_p} - |t_i^p| = \lambda-2$
(after identifying the $\lambda^k$ suspensions with
$\lambda$ for $(k,p)=1$),
and the transfer of $1$ into this degree is $\theta$, so we have
	\[
	pN(t_i) = \theta t_i^p = 0.
	\]
\end{proof}

\begin{proof}[Proof of Lemma \ref{lem:thetat=0}]
By Lemma \ref{lem:mult-in-bs1}, 
we have $e_1^p \doteq \theta e_p$ so that
$\theta t_1 = \sigma(\theta e_p) = 0$, since
$\sigma$ annihilates decomposables. For the remaining classes,
consider the commutative diagram
	\[
	\xymatrix{
	K(\underline{\mathbb{Z}}, \lambda)_+\ar[d]_{[\theta]} \ar[r] & 
	\Sigma^{\lambda}\underline{\mathbb{Z}} \ar[d]^{\theta}\\
	K(\underline{\mathbb{Z}}, 2)_+ \ar[r] &
	\Sigma^2\underline{\mathbb{Z}}
	}
	\]
where $[\theta] = \Omega^{\infty}(\theta)$. Thus, to show that
$\theta t_i = 0$ for $i\ge 2$, it is enough to show that
$[\theta]_*e_{p^i}$ is decomposable
in $\pi_{\star}(\underline{\mathbb{Z}}_{(p)} \otimes K(\underline{\mathbb{Z}},2)_+)$ for $i\ge 2$.

Write
	\[
	\underline{\mathbb{Z}}_{(p)} \otimes
	K(\underline{\mathbb{Z}}, 2)_+
	= 
	\underline{\mathbb{Z}}_{(p)}\{\gamma_i(\beta_1)\}
	\]
where the elements $\gamma_i(\beta_1)$ are the standard module
generators of $H_*(\mathbb{C}P^{\infty}; \mathbb{Z})$,
and write $\beta_{(i)} = \gamma_{p^i}\beta_1$.
To show that $[\theta]_*(e_{p^i})$ is decomposable
for $i\ge 2$, it is enough to establish the following
two claims:
	\begin{enumerate}[(a)]
	\item $[\theta]_*(e_{p^i}) \doteq 
	\frac{p^{i-1}\theta}{u_{\lambda}^{p^i(p-1)-1}}\beta_{(i)}$, and
	\item $\beta_{(i-1)}^p \doteq p\beta_{(i)}$.
	\end{enumerate}
Claim (b) is just the classical computation of the product
in homology for $H_*(\mathbb{C}P^{\infty}, \mathbb{Z})$.
For claim (a), let $\iota_{\lambda}$ denote the fundamental
class in cohomology for $K(\underline{\mathbb{Z}}, \lambda)$
and $\iota_2$ the same for $K(\underline{\mathbb{Z}}, 2)$.
Then we have $[\theta]^*(\iota_2) = \theta \iota_{\lambda}$
by design, and hence
	\[
	[\theta]^*(\iota_2^j) = \theta^j \iota_{\lambda}^j. 
	\]
The map on homology is now determined by the relation
	\[
	\langle [\theta]_*e_{p^i}, \iota_2^j\rangle
	= \theta^j\langle e_{p^i}, \iota^j_{\lambda}\rangle
	\in \pi_{\star}\underline{\mathbb{Z}}_{(p)}.
	\]
Since $\theta^j$ is a transferred class, the value above
is also a transfer, and hence determined by its restriction
to an underlying class. But $\mathrm{res}([\theta]) = [p]$
and we clearly have $[p]_*(\mathrm{res}(e_{p^i})) =p^i\beta_{(i)}$,
which agrees with the restriction of 
$\frac{p^{i-1}\theta}{u_{\lambda}^{p^i(p-1)-1}}\beta_{(i)}$. 
This completes the proof.
\end{proof}

\section{Digression: Detecting equivalences nonequivariantly}\label{sec:detect-equivalences}

The goal of this section is to establish a criterion
for detecting equivalences of $\underline{\mathbb{Z}}$-modules.
We recall that
	\[
	\underline{\mathbb{Z}}^{\Phi C_p} \simeq
	\mathbb{F}_p[b]
	\]
where the class $b$ in degree 2 arises from taking the geometric
fixed points of the Thom class $u_{\lambda}: S^{\lambda} \to \Sigma^2\underline{\mathbb{Z}}$.

\begin{proposition}\label{prop:detect-equiv} Let $f: M \to N$ be a map
of $\underline{\mathbb{Z}}$-modules which
are bounded below. Assume the 
following conditions are satisfied:
	\begin{enumerate}[{\rm (i)}]
	\item $f$ is an underlying equivalence.
	\item $\pi_jM^{\Phi C_p}$ and $\pi_jN^{\Phi C_p}$
	are finite dimensional of the same rank, for all $j$.
	\item $\pi_*M^{\Phi C_p}$ and $\pi_*N^{\Phi C_p}$
	are graded-free $\mathbb{F}_p[b]$-modules.
	\end{enumerate}
Then $f$ is an equivalence.
\end{proposition}

We will deduce this proposition from the following one,
which relates geometric and Tate fixed points.

\begin{proposition}\label{prop:tate-from-geo} Let $M$ be a
$\underline{\mathbb{Z}}$-module which is both
bounded above and below. Then the natural map
	\[
	M^{\Phi C_p}[b^{-1}] \to M^{tC_p}
	\]
is an equivalence.
\end{proposition}

\begin{proof}[Proof of Proposition \ref{prop:detect-equiv}
assuming Proposition \ref{prop:tate-from-geo}]
By assumption (i), it is enough to check that
$f^{\Phi C_p}$ is an equivalence; by assumption
(ii), it is enough to check that $\pi_*(f^{\Phi C_p})$ is
an injection; and by assumption (iii) it is enough
to check that $\pi_*(f^{\Phi C_p})[b^{-1}]$ is an injection.

Again by (i), the map $f^{tC_p}$ is an equivalence. So, from the diagram
	\[
	\xymatrix{
	M^{\Phi C_p}[b^{-1}] \ar[r]\ar[d] & N^{\Phi C_p}[b^{-1}]\ar[d]\\
	M^{tC_p} \ar[r]_{\simeq} & N^{tC_p}
	}
	\]
we see that it is enough to check that the vertical maps
are injective on homotopy. More generally, we show that whenever
$X$ is a bounded below $\underline{\mathbb{Z}}$-module,
the map
	\[
	\pi_* X^{\Phi C_p}[b^{-1}] \to \pi_*X^{tC_p}
	\]
is injective. Indeed, by Proposition \ref{prop:tate-from-geo}
and the fact that the Tate construction commutes with
limits of Postnikov towers (see, e.g., \cite[I.2.6]{nikolaus-scholze}),
we have
	\[
	\lim_n \left((\tau_{\le n}X)^{\Phi C_p}[b^{-1}]\right)
	\stackrel{\simeq}{\to}
	\lim_n (\tau_{\le n}X)^{tC_p} \simeq X^{tC_p}.
	\]
Therefore, we need only check that
	\[
	\pi_*X^{\Phi C_p}[b^{-1}] \to 
	\pi_*\lim_n \left((\tau_{\le n}X)^{\Phi C_p}[b^{-1}] \right)
	\]
is injective. Since the maps $X^{\Phi C_p} \to (\tau_{\le n}X)^{\Phi C_p}$
have increasingly connective fibers, we can replace the left hand
side by $(\lim_n \pi_*(\tau_{\le n}X)^{\Phi C_p})[b^{-1}]$ and reduce
to showing that
	\[
	(\lim_n \pi_*(\tau_{\le n}X)^{\Phi C_p})[b^{-1}]
	\to \lim_n \pi_*\left((\tau_{\le n}X)^{\Phi C_p}[b^{-1}]\right)
	\]
is injective. Finally, this reduces to showing that the kernel of
	\[
	\lim_n \pi_*(\tau_{\le n}X)^{\Phi C_p}
	\to \lim_n \pi_*\left((\tau_{\le n}X)^{\Phi C_p}[b^{-1}]\right)
	\]
consists of elements annihilated by a power of $b$. This is clear because,
for each $j$, the system $\{\pi_j(\tau_{\le n}X)^{\Phi C_p}\}_n$
is eventually constant.
\end{proof}

\begin{proof}[Proof of Proposition \ref{prop:tate-from-geo}]
Let $\mathcal{E}$ denote the full subcategory of
$\underline{\mathbb{Z}}$-modules $M$ for which
	\[
	M^{\Phi C_p}[b^{-1}] \to M^{tC_p}
	\]
is an equivalence. Then $\mathcal{E}$ is stable, closed under
retracts, and closed under suspending by representation
spheres.

The map $M^{\Phi C_p}[b^{-1}] \to M^{tC_p}$ is one
of $\underline{\mathbb{Z}}^{\Phi C_p}=\mathbb{F}_p[b]$-modules,
and hence one of $\mathbb{F}_p$-modules,  so it must be a retract of
	\[
	(M/p)^{\Phi C_p}[b^{-1}] = M^{\Phi C_p}[b^{-1}]/p
	\to M^{tC_p}/p = (M/p)^{tC_p}.
	\]
Thus $M/p \in \mathcal{E}$ if and only if $M \in \mathcal{E}$. 
So, by replacing $M$ with $M/p$ and
considering the Postnikov tower, we are
reduced to proving the proposition in the case where
$M \in \mathsf{Mod}^{\heartsuit}_{\underline{\mathbb{Z}}}$
is a Mackey functor which is a module over $\underline{\mathbb{F}}_p$.

In particular, $M^e$ is an $\mathbb{F}_p[C_p]$-module. 
Let $\gamma$ denote the generator of $C_p$ so that
$\mathbb{F}_p[C_p] = \mathbb{F}_p[\gamma]/(1-\gamma)^p$. 
Let $F_jM \subseteq M$ be the sub-Mackey functor generated
by $(1-\gamma)^jM^e \subseteq M^e$. This is a finite filtration
with associated graded pieces given by Mackey functors
with trivial underlying action. So, since $\mathcal{E}$ is
a thick subcategory, we are reduced to the case when 
$M$ is a discrete $\underline{\mathbb{F}}_p$-module with
trivial underlying action.

For the next reduction we recall some notation. If $N$ is
any Mackey functor, denote by $N_{C_p}$ the
Mackey functor $N\otimes C_{p+}$ and, if $A$ is an
abelian group, denote by $\underline{A}_{\mathrm{tr}}$
the Mackey functor whose transfer map is the identity on $A$
and whose restriction map is multiplication by $p$. 
We also recall that the transfer extends to a map of 
Mackey functors $\mathrm{tr}: N_{C_p} \to N$. 

Now consider the two exact sequences
	\[
	0 \to \mathrm{im}(\mathrm{tr}) \to M \to
	M/\mathrm{im}(tr) \to 0
	\]
	\[
	0 \to \mathrm{ker}(\mathrm{tr}) \to \underline{M}^e_{\mathrm{tr}} \to
	\mathrm{im}(\mathrm{tr}) \to 0
	\]

If $N$ is any Mackey functor with $N^e = 0$, then
$N \in \mathcal{E}$ since then $N=N^{\Phi C_p}$ is bounded above
and hence $N^{\Phi C_p}[b^{-1}]=0$. 
Thus, from the exact sequences above,
we are reduced to the case where $M$ is of the form
$\underline{V}_{\mathrm{tr}}$ for an
$\mathbb{F}_p$-vector space $V$ (with trivial action).
Now recall that $(\underline{\mathbb{F}}_p)_{\mathrm{tr}}
= \Sigma^{2-\lambda}\underline{\mathbb{F}}_p$
and hence $\underline{V}_{\mathrm{tr}} = \Sigma^{2-\lambda}\underline{V}$.
So we are reduced to showing that the constant Mackey functor
$\underline{V}$ lies in $\mathcal{E}$, where $V$
is an $\mathbb{F}_p$-vector space with trivial action.
This certainly holds for $V=\mathbb{F}_p$, and in general we have
	\[
	\underline{V}^{\Phi C_p} \simeq 
	\underline{\mathbb{F}}^{\Phi C_p}_p \otimes_{\mathbb{F}_p} V,
	\]
since geometric fixed points commutes with colimits, and
	\[
	V^{tC_p} \simeq \mathbb{F}_p^{tC_p} \otimes_{\mathbb{F}_p}V
	\]
by direct calculation. (Notice this holds even when $V$
is infinite-dimensional). This completes the proof.
\end{proof}

\section{Proof of the main theorem}\label{sec:proof}

We are now ready to prove the main theorem. Recall that
we have constructed classes
	\[
	t_i \in \pi_{2p^{i-1}\rho_{C_p} - \lambda}(
	\underline{\mathbb{Z}}_{(p)} \otimes
	\underline{\mathbb{Z}}_{(p)}),
	\]
and shown that $\theta t_i = 0$ and $pN(t_i) = 0$.
With notation as in the introduction, let
	\[
	X_i = \left( S^0 \oplus (S^0[Nt_i] \otimes T_{\theta}(t_i))\right)
	\]
and
	\[
	X = \bigotimes_{i\ge 1}\left( S^0 \oplus (S^0[Nt_i] \otimes T_{\theta}(t_i))\right)
	\]
Then, choosing nullhomotopies which witness $\theta t_i = 0$,
we get a map:
	\[
	f: \underline{\mathbb{Z}}_{(p)} \otimes
	\bigotimes_{i\ge 1}
	\left( S^0 \oplus (S^0[Nt_i] \otimes T_{\theta}(t_i))\right)
	\longrightarrow 
	\underline{\mathbb{Z}}_{(p)} \otimes \underline{\mathbb{Z}}_{(p)}
	\]

The main theorem is then the statement:

\begin{theorem} The map $f$ is an equivalence.
\end{theorem}
\begin{proof} Combine Proposition \ref{prop:detect-equiv}
with the two lemmas below.
\end{proof}

\begin{lemma} The map $f^e$ is an underlying equivalence.
\end{lemma}
\begin{proof} First observe that, by our construction
in the proof of Lemma \ref{lem:pn=0},
the map $\widehat{N(t_i)}$ restricts to the map
$t_i^{p-1}\hat{t}_i$, since the nullhomotopy witnessing
$pN(t_i) = 0$ was chosen to restrict to the nullhomotopy
chosen for $pt_i^p$ that came from the already chosen
nullhomotopy of $pt_i$. The upshot is that the map
	\[
	S^0 \oplus S^0[Nt_i] \otimes T_{\theta}(t_i) \to
	\underline{\mathbb{Z}} \otimes \underline{\mathbb{Z}}
	\]
restricts on underlying spectra to the map
	\[
	S^0[t_i]/(pt_i) \to 
	\mathbb{Z} \otimes \mathbb{Z}
	\]
obtained just from the relation $pt_i = 0$ and extended
via the multiplicative structure.

In particular, on mod $p$ homology $f^e$ induces a ring map
	\[
	\mathbb{F}_p[t_i] \otimes \Lambda(x_i) \to 
	\mathbb{F}_p[\xi_i] \otimes \Lambda(\tau_i).
	\]
We know that $t_i$ maps to $\xi_i$ and that
$\beta x_i = t_i$, so that $\beta(f_*^e(x_i)) = \xi_i$.
Modulo decomposables, $\tau_i$ is
the only element whose Bockstein
is $\xi_i$. So $x_i$ must map to $\tau_i$, mod decomposables.
It follows that $f^e$ is a mod $p$ equivalence, and hence
an equivalence.
\end{proof}

\begin{lemma} $(\underline{\mathbb{Z}} \otimes X)^{\Phi C_p}$
and $(\underline{\mathbb{Z}} \otimes \underline{\mathbb{Z}})^{\Phi C_p}$
are free $\mathbb{F}_p[b]$-modules, finite-dimensional in each degree,
and isomorphic as graded vector spaces over $\mathbb{F}_p$.
\end{lemma}
\begin{proof} If $Y$ is any $C_p$-spectrum, then
	\[
	(\underline{\mathbb{Z}}_{(p)} \otimes Y)^{\Phi C_p}
	=
	\mathbb{F}_p[b] \otimes Y^{\Phi C_p}
	\simeq
	\mathbb{F}_p[b] \otimes_{\mathbb{F}_p} 
	(\mathbb{F}_p \otimes Y^{\Phi C_p})
	\]
is a free $\mathbb{F}_p[b]$-module. Applying this in the cases
$Y=X$ and $Y=\underline{\mathbb{Z}}$, we see that each
is a free $\mathbb{F}_p[b]$, evidently finite-dimensional
in each degree. So it suffices to prove that
	\[
	\mathbb{F}_p\otimes X^{\Phi C_p} \cong
	\mathbb{F}_p\otimes (\mathbb{F}_p[b])
	\]
as graded vector spaces. Notice that we can write,
\emph{as graded vector spaces},
	\[
	\mathbb{F}_p \otimes X_i^{\Phi C_p}
	\cong \mathbb{F}_p[d_{(i-1)}, \xi_i] \otimes_{\mathbb{F}_p} 
	\Lambda(\sigma_{i-1}, \tau_i)/
	(d_{(i-1)}^p, d_{(i-1)}\tau_i, d_{(i-1)}^{p-1}\sigma_{i-1},
	\sigma_{i-1}\tau_i),
	\]
where $|\sigma_{i-1}| = 2p^{i-1}-1$ and $|d_{(i-1)}|=2p^{i-1}$.
Indeed, $\hat{t}_i$, on geometric fixed points, gives rise to
two classes; one we are calling $d_{(i-1)}$ and the other 
we are calling $\sigma_{i-1}$. Similarly, $\widehat{N(t_i)}$,
on geometric fixed points, gives rise to two classes: one
we are calling $\xi_i$ and the other $\tau_i$, in their usual
degrees. The relations are the ones needed to ensure
that the monomials not arising from geometric fixed points
of elements in $X_i$ are omitted. 

It follows that we have an isomorphism of graded vector spaces
	\[
	\mathbb{F}_p \otimes X^{\Phi C_p} \cong 
	\mathbb{F}_p[\xi_n: n\ge 1]
	\otimes_{\mathbb{F}_p} \mathbb{F}_p[d_{(i)} : i\ge 0] 
	\otimes_{\mathbb{F}_p}
	\Lambda(\sigma_j, \tau_k: j\ge 0, k\ge 1)/
	(d_{(i)}^p, d_{(i-1)}\tau_i, d_{(i)}^{p-1}\sigma_{i},
	\sigma_{i-1}\tau_i).
	\]
We are trying to show that this is isomorphic, as a graded vector
space to
	\[
	\mathbb{F}_p \otimes \mathbb{F}_p[b]
	\cong
	\mathbb{F}_p[\xi_n:n\ge 1] \otimes_{\mathbb{F}_p} \Lambda(\tau_i:i\ge 0)
	\otimes_{\mathbb{F}_p} \mathbb{F}_p[b].
	\]
We may regard each vector space as a module over
$\mathbb{F}_p[\xi_n:n\ge 0]$ in the evident way, and hence
reduce to showing that the two vector spaces
	\[
	V= \Lambda(\tau_i:i\ge 0)
	\otimes_{\mathbb{F}_p} \mathbb{F}_p[b]
	\]
and 
	\[
	W= \mathbb{F}_p[d_{(i)} : i\ge 0] \otimes_{\mathbb{F}_p}
	\Lambda(\sigma_j, \tau_k: j\ge 0, k\ge 1)/
	(d_{(i)}^p, d_{(i-1)}\tau_i, d_{(i)}^{p-1}\sigma_{i},
	\sigma_{i-1}\tau_i)
	\]
are isomorphic. (Here recall that $|\sigma_i|=|\tau_i|=2p^i-1$,
$|b|=2$, and $|d_{(i)}|=2p^i$).

Let $I$ range over sequences $(a_0, a_1, ...)$ with
$0\le a_i\le p-2$, $J$ range
over sequences $(\varepsilon_0, \varepsilon_1, ...)$
with $\varepsilon_i \in \{0,1\}$, $K$
range over sequences $(\kappa_0, \kappa_1, ...)$
with $\kappa_i \in\{0,1\}$, and let $K'$ 
range over sequences $(\kappa'_0, \kappa'_1, ...)$
with $\kappa'_i \in \{0,1\}$. 
 We impose the following requirements
 on these sequences:
 	\begin{itemize}
	\item Each sequence has finite support.
	\item If $\kappa'_i = 1$, then $\kappa_i = 1$.
	(So $K'$ is otained from $K$ by changing some
	subset of $1$s to $0$s).
	\item $J\cdot K = I\cdot K = (0,0,...)$. That is:
	$I$ and $K$ have disjoint support and $J$ and $K$
	have disjoint support.
	\end{itemize}
Then $V$
has a basis of monomials
	\[
	M_{I, J,K} =
	(\prod_{i\ge 0} b^{a_ip^i})\tau_J(\prod_{i\ge 0}
	b^{\kappa_i(p-1)p^i})\tau_{K'}
	\]
and $W$ has a basis of monomials
	\[
	N_{I,J,K} = d_I\sigma_J(\prod_{i\ge 0}
	d_{(i)}^{(\kappa_i-\kappa'_i)(p-1)})\tau_{K'[1]}
	\]
where $K'[1] = (0, \kappa'_0, \kappa'_1, ...)$.
These have the same number of basis elements in each
dimension, so $V \cong W$.
\end{proof}

\bibliographystyle{amsalpha}
\bibliography{Bibliography}

\providecommand{\bysame}{\leavevmode\hbox to3em{\hrulefill}\thinspace}
\providecommand{\MR}{\relax\ifhmode\unskip\space\fi MR }
\providecommand{\MRhref}[2]{%
  \href{http://www.ams.org/mathscinet-getitem?mr=#1}{#2}
}
\providecommand{\href}[2]{#2}
\begin{thebibliography}{HHR16}

\bibitem[Car99]{caruso}
Jeffrey~L. Caruso, \emph{Operations in equivariant {${\bf Z}/p$}-cohomology},
  Math. Proc. Cambridge Philos. Soc. \textbf{126} (1999), no.~3, 521--541.
  \MR{1684248}

\bibitem[FL04]{ferland-lewis}
Kevin~K. Ferland and L.~Gaunce Lewis, Jr., \emph{The {$R{\rm O}(G)$}-graded
  equivariant ordinary homology of {$G$}-cell complexes with even-dimensional
  cells for {$G={\Bbb Z}/p$}}, Mem. Amer. Math. Soc. \textbf{167} (2004),
  no.~794, viii+129. \MR{2025457}

\bibitem[Gre88]{greenlees}
J.~P.~C. Greenlees, \emph{Stable maps into free {$G$}-spaces}, Trans. Amer.
  Math. Soc. \textbf{310} (1988), no.~1, 199--215. \MR{938918}

\bibitem[HHR16]{hhr}
M.~A. Hill, M.~J. Hopkins, and D.~C. Ravenel, \emph{On the nonexistence of
  elements of {K}ervaire invariant one}, Ann. of Math. (2) \textbf{184} (2016),
  no.~1, 1--262. \MR{3505179}

\bibitem[HK01]{hu-kriz}
Po~Hu and Igor Kriz, \emph{Real-oriented homotopy theory and an analogue of the
  {A}dams-{N}ovikov spectral sequence}, Topology \textbf{40} (2001), no.~2,
  317--399. \MR{1808224}

\bibitem[HW20]{hahn-wilson}
Jeremy Hahn and Dylan Wilson, \emph{Eilenberg--{M}ac {L}ane spectra as
  equivariant {T}hom spectra}, Geom. Topol. \textbf{24} (2020), no.~6,
  2709--2748. \MR{4194302}

\bibitem[Lew88]{lewis}
L.~Gaunce Lewis, Jr., \emph{The {$R{\rm O}(G)$}-graded equivariant ordinary
  cohomology of complex projective spaces with linear {${\bf Z}/p$} actions},
  Algebraic topology and transformation groups ({G}\"{o}ttingen, 1987), Lecture
  Notes in Math., vol. 1361, Springer, Berlin, 1988, pp.~53--122. \MR{979507}

\bibitem[NS18]{nikolaus-scholze}
Thomas Nikolaus and Peter Scholze, \emph{On topological cyclic homology}, Acta
  Math. \textbf{221} (2018), no.~2, 203--409. \MR{3904731}

\bibitem[Oru89]{oruc}
Melda~Yaman Oru\c{c}, \emph{The equivariant {S}teenrod algebra}, Topology Appl.
  \textbf{32} (1989), no.~1, 77--108. \MR{1003301}

\bibitem[San19]{sankar}
Krishanu~Roy Sankar, \emph{Steinberg summands in the free $\mathbb{F}_p$-module
  on the equivariant sphere spectrum}, 2019.

\bibitem[Voe03]{voevodsky}
Vladimir Voevodsky, \emph{Reduced power operations in motivic cohomology},
  Publ. Math. Inst. Hautes \'{E}tudes Sci. (2003), no.~98, 1--57. \MR{2031198}

\end{thebibliography}

\end{document}